\documentclass[a4paper,12pt]{article}

\newenvironment{proof}{\noindent {\bf Proof:}}{\hfill $\Box$}

\newtheorem{lemma}{Lemma}
\newtheorem{proposition}{Proposition}

\usepackage{booktabs}
\usepackage[T1]{fontenc}

\usepackage{amsmath}
\usepackage{amssymb}
\usepackage{amsfonts}
\usepackage{mathrsfs} 
\usepackage{graphicx}
\usepackage{color}
\usepackage{ulem}
\usepackage{bm}
\usepackage{hyperref}

\newcommand{\R}{\mathbb{R}} 
\newcommand{\N}{\mathbb{N}} 

\title{\bf Infinite-dimensional Christoffel-Darboux polynomial kernels on Hilbert spaces}

\begin{document}

\author{Didier Henrion$^{1,2}$}

\footnotetext[1]{CNRS; LAAS; Universit\'e de Toulouse, 7 avenue du colonel Roche, F-31400 Toulouse, France. }
\footnotetext[2]{Faculty of Electrical Engineering, Czech Technical University in Prague, Technick\'a 2, CZ-16626 Prague, Czechia.}

\date{Draft of \today}

\maketitle

In these notes, the Christoffel-Darboux polynomial kernel \cite{lpp22} is extended to infinite-dimensional Hilbert spaces, following the finite-dimensional treatment of \cite[Part 1]{d22}, itself inspired from \cite{p20} and \cite{lpp22}.

\section{Preliminaries}

\subsection{Separable Hilbert space}

Let $H$ be a separable real Hilbert space equipped with an inner product $\langle .,.\rangle$ and let $(e_k)_{k=1,2,\ldots}$ be a complete orthonormal system in $H$ with $e_1=1$, see e.g. \cite[Section 16.3]{rf10}. Given $n \in \N$,
consider the projection mapping
\[\begin{array}{llll}
\pi_n : & H & \to & H \\
& x & \mapsto & \sum_{k=1}^n \langle x,e_k \rangle e_k.
\end{array}\]
In particular, note that \[x = \lim_{n\to\infty} \pi_n(x) = \sum_{k=1}^{\infty} \langle x,e_k \rangle e_k.\]
Also note that
\[
|\pi_n(x)|^2 = \sum_{k=1}^{n} \langle x, e_k \rangle^2
\]
and
\begin{equation}\label{np}
|x|^2= \lim_{n\to\infty} |\pi_n(x)|^2 = \sum_{k=1}^{\infty} \langle x, e_k \rangle^2.
\end{equation}

\subsection{Polynomials}

Let $c_0(\N)$ denote the set of integer sequences with finitely many non-zero elements, i.e. if $a=(a_1,a_2,\ldots) \in c_0(\N)$ then $\text{card} \{k : a_k \neq 0\}< \infty$. Let us define the {\it monomial} of degree $a \in c_0(\N)$ as
\[
x^a := \prod_{k=1}^{\infty} {\langle x, e_k \rangle}^{a_k}.
\]
This is a product of finitely many powers of linear functionals.
{\it Polynomials} in $H$ are defined as linear combinations of monomials
\[\begin{array}{llll}
p : & H & \to & \R \\
& x & \mapsto & \sum_{a \in \text{spt}\:p} p_a x^a
\end{array}
\]
with coefficients $p_a \in \R$ indexed in the support $\text{spt}\:p \subset c_0(\N)$.
In the infinite dimensional case, the notion of degree of a polynomial is twofold.
The {\it algebraic degree} of $p(.)$ is \[d:=\max_{a \in \text{spt}\:p} \sum_k\:a_k\]
and it corresponds to the classical notion of total degree in the finite dimensional case. The {\it harmonic degree} of $p(.)$ is \[n:=\max_{a \in \text{spt}\:p} \{k \in \N: a_k \neq 0\}\]
and it corresponds to the number of variables in the finite dimensional case.

Given $d,n \in \N$, let $P_{d,n}$ denote the finite-dimensional vector space of polynomials of algebraic degree up to $d$ and harmonic degree up to $n$. Its dimension is the binomial coefficient ${n+d\choose n}$. Each polynomial $p(.) \in P_{d,n}$ can be identified with its coefficient vector $p:=(p_a)_{a \in \text{spt}\:p} \in \R^{n+d\choose n}$. For example, if $d=4$ and $n=2$, then $P_{4,2}$ has dimension  ${6\choose 2}=15$. The monomial of degree $a=(3,1,0,0, \ldots)$ is $\langle x,e_1\rangle^3\langle x,e_2\rangle$. It belongs to $P_{4,2}$ since it has algebraic degree $4$ and harmonic degree $2$.

\subsection{Compactness}

A useful characterization of compact sets in separable Hilbert spaces is as follows \cite[item 45, p.346, IV.13.42]{ds58}. 

\begin{proposition}\label{compact}
	A closed set $X \in H$ is compact if and only if for all $\epsilon \in \R$ there exists $n \in \N$ such that $\sup_{x \in X} |x|^2-|\pi_n(x)|^2 < \epsilon^2$.
\end{proposition}

Compared to the finite-dimensional case, this enforces additional conditions on the polynomials defining compact semialgebraic sets.

\begin{proposition}
The ellipsoid
\begin{equation}\label{ellipsoid}
X=\{x \in H : \sum_{k=1}^{\infty} p_k \langle x,e_k \rangle^2 \leq 1\}
\end{equation}
is compact if the sequence  $(p_k)_{k=1,2,\ldots}$ is strictly positive and strictly increasing. For example one may choose $p_k = k$.
\end{proposition}

\begin{proof}
$X$ is closed and bounded since $p_k>0$ for all $k$.
For $x \in X$ arbitrary, let $c_k(x):=\langle x,e_k \rangle$ and $s_k(x):=\sum_{l=k}^{\infty} c^2_l(x)$ for all $k$. Let us use Proposition \ref{compact} by proving that for all $\epsilon \in \R$ there exists $n \in \N$ such that $\sup_{x\in X} s_{n+1}(x)  < \epsilon^2$.

Let $q_1:=p_1$ and $q_{k+1}:=p_{k+1}-p_k$, so that $q_k > 0$ for all $k$. Then
\[\begin{array}{lll}
1 & \geq & p_1 c_1(x) + p_2 c_2(x) + p_3 c_3(x) + \cdots \\
& = & q_1 c_1(x) + (q_1 + q_2) c_2(x) + (q_1 + q_2 + q_3) c_3(x) + \cdots \\
& = & q_1 (c_1(x)+c_2(x)+c_3(x)+\ldots) + q_2 (c_2(x)+c_3(x)+\ldots) + q_3 (c_3(x)+\cdots) + \cdots \\
& = & q_1 s_1(x) + q_2 s_2(x) + q_3 s_3(x) + \cdots
\end{array}\]
It follows that the sequence $(q_k s_k(x))_{k=1,2,\ldots}$ is non-negative and summing up at most to one. So for each $\epsilon$, we can find $n$ such that $q_{n+1} s_{n+1}(x)$ is small enough, and in particular smaller than $q_{n+1} \epsilon^2$.
\end{proof}

The Sobolev space of functions whose derivatives up to order $m>0$ are square integrable is a separable Hilbert space \cite[Thm. 3.6]{af03}. In this space, the squared norm of an element is $\sum_{k=1}^{\infty} k^m \langle x,e_k \rangle^2$ and unit balls are therefore precisely of the form \eqref{ellipsoid}.

\begin{proposition}
	The Hilbert cube
	\[
	X=\{x \in H : |\langle x,e_k \rangle| \leq \frac{1}{k}, \: k=1,2,\ldots´ \}
	\]
	is compact.

\end{proposition}

\begin{proof}
	See \cite[Item 70 p. 350]{ds58}. If $x \in X$ then $\sum_{k=n+1}^{\infty} \langle x,e_k \rangle^2 \leq \sum_{k=n+1}^{\infty} \frac{1}{k^2}$ can be made arbitrarily small for sufficient large $n$ and we can use Proposition \ref{compact}.
\end{proof}
 
Let $C(X)$ denote the space of continuous functions on $X$, and let $P(X) \subset C(X)$ denote the space of polynomials on $X$.

\begin{proposition}\label{density}
If $X$ is compact then $P(X)$ is dense in $C(X)$.
\end{proposition}
\begin{proof}
	Observe that the set  of polynomials on $X$ is an algebra (i.e. the product of two elements of $P(X)$ is an element of $P(X)$) that separates points (i.e. for all $x_1 \neq x_2 \in X$, there is a $p \in P(X)$ such that $p(x_1) \neq p(x_2)$) and that contains constant functions (corresponding to the degree $a=0$). The 
	 result follows from the Stone-Weierstrass Theorem \cite[Section 12.3]{rf10}.
	\end{proof}

\subsection{Moments}

Measures on topological spaces are defined as countably additive non-negative functions acting on the sigma algebra, the collection of subsets which are closed under complement and countable unions \cite[Section 17.1]{rf10}.
Measures on a separable Hilbert space $H$ are uniquely determined by their actions on test functions.

\begin{proposition}\cite[Prop. 1.5]{d06}\label{match}
	Let $\mu_1$ and $\mu_2$ be measures on $H$ such that
	\[\int f(x)d\mu_1(x) = \int f(x)d\mu_2(x)\] for all continuous bounded functions $f$ on $H$. Then $\mu_1=\mu_2$.
\end{proposition}

Given a measure $\mu$ supported on a subset $X$ of $H$, and given $a \in c_0(\N)$, the {\it moment} of $\mu$ of degree $a$ is the number
\[
\int_X x^a d\mu(x).
\]
A measure on a compact set is uniquely determined by its sequence of moments.
\begin{proposition}
	Let $\mu_1$ and $\mu_2$ be measures on a compact set $X$ satisfying
	\[
	\int_X x^a d\mu_1(x) = \int_X x^a d\mu_2(x)
	\]
	for all $a \in c_0(\N)$. Then $\mu_1= \mu_2$.
\end{proposition}
\begin{proof}
If all moments of $\mu_1$ and $\mu_2$ coincide, then 
\[
\int_X p(x) d\mu_1(x) = \int_X p(x) d\mu_2(x)
\]	
for all polynomials on $X$, and by Proposition \ref{density} for all continuous functions on $X$.
Since continuous functions are bounded on a compact set $X$, we can use Proposition \ref{match} to conclude.
\end{proof}


\section{Christoffel-Darboux polynomial}

Let $\mu$ be a given probability measure on a given compact set $X$ of $H$.
Given $d,n \in \N$, the finite-dimensional vector space $P_{d,n}$ of polynomials of algebraic degree up to $d$ and harmonic degree up to $n$ is a Hilbert space once equipped with the inner product \[\langle p,q \rangle_{\mu}:=\int p(x)q(x)d\mu(x).\] Let $b_{d,n}(.)$ denote a basis for $P_{d,n}$, of dimension $n+d\choose n$. Any element $p \in P_{d,n}$ 
can be expressed as $p(.) = p^T b_{d,n}(.)$ with $p$ a vector of coefficients.

Let
\[
M^{\mu}_{d,n} := \int_X b_{d,n}(x)b_{d,n}(x)^T d\mu(x)
\]
be the moment matrix of order $(d,n)$ of measure $\mu$, which is the Gram matrix of the inner products of pairwise entries of vector $b_{d,n}(.)$. This matrix is positive semi-definite of size $n+d\choose n$.
If it is non-singular, then it can be written as
\[
M^{\mu}_{d,n} = Q S Q^T = \sum_{i=1}^{n+d\choose n} s_i q_i q^T_i 
\]
with $q_i(x):=q^T_i b_{d,n}(x)$ and $q^T_i M^{\mu}_{d,n} q_j=s_i$ if $i=j$ and $0$ if $i \neq j$.

The Christoffel-Darboux (CD) kernel is then defined as
\[
K^{\mu}_{d,n}(x,y):= \sum_{i=1}^{n+d\choose n} s^{-1}_i q_i(x) q_i(y) = 
b^T_{d,n}(x)(M^{\mu}_{d,n})^{-1}b_{d,n}(y)
\]
and the CD polynomial is defined as the sum of squares diagonal evaluation of the kernel
\[
p^{\mu}_{d,n}(x) := K^{\mu}_{d,n}(x,x).
\]

\begin{lemma}
The vector space $P_{d,n}$ equipped with the CD kernel is a RKHS (reproducible kernel Hilbert space).
\end{lemma}

\begin{proof}
The linear functional $p \mapsto \langle p(.), K^{\mu}_d(.,y)\rangle_{\mu}$ has the reproducing property:
\begin{align*}
\forall p \in P_{d,n}, \quad\langle p(.), K^{\mu}_{d,n}(.,y)\rangle_{\mu} & = \int_X p(x)K^{\mu}_{d,n}(x,y)d\mu(x) \\
& = \int_X p(x)b_{d,n}(x)^T (M^{\mu}_{d,n})^{-1} b_{d,n}(y) d\mu(x)\\
& = p^T \left(\int_X b_{d,n}(x) b_{d,n}(x)^T d\mu(x)\right) (M^{\mu}_{d,n})^{-1} b_{d,n}(y) \\
& = p^T b_{d,n}(y) = p(y)
\end{align*}
and it is continuous:
\[
\forall p \in P_{d,n}, \quad \langle p(.), K^{\mu}_{d,n}(.,y)\rangle_{\mu}^2 \leq \int_X p(x)^2d\mu(x) \int_X K^{\mu}_d(x,x)d\mu(x).
\]
\end{proof}

\section{Christoffel function}

The Christoffel function is defined as
\begin{equation}\label{cf}
\begin{array}{llll}
C^{\mu}_{d,n} \: : \: & H  & \to & [0,1] \\
& z  & \mapsto &  \displaystyle \min_{p \in P_{d,n}} \int_X p^2(x)d\mu(x) \quad\text{s.t.}\quad p(z)=1.
\end{array}
\end{equation}

\begin{lemma}
For each $z \in H$, the minimum is
\[
C^{\mu}_{d,n}(z) = \frac{1}{K^{\mu}_{d,n}(z,z)} = \frac{1}{p^{\mu}_{d,n}(z)}
\]
and it is achieved at
\[
p(.) = \frac{K^{\mu}_{d,n}(.,z)}{K^{\mu}_{d,n}(z,z)}.
\]
\end{lemma}

\begin{proof}
It holds
\begin{align*}
1 = p^2(z) & = (\int_X K^{\mu}_{d,n}(x,z)p(x)d\mu(x))^2 \\
& \leq \int_X K^{\mu}_{d,n}(x,z)^2 d\mu(x)
\int_X p^2(x)d\mu(x) \\
& = K^{\mu}_{d,n}(z,z) \int_X p^2(x)d\mu(x)
\end{align*}
so
\[C^\mu_{d,n}(z) \geq \frac{1}{K^\mu_{d,n}(z,z)}.\]
Now observe that the polynomial
\[p(.):=\frac{K^{\mu}_{d,n}(.,z)}{K^{\mu}_{d,n}(z,z)} \in P_{d,n}
\]
is admissible for problem \eqref{cf}, i.e. $p(z)=1$ and hence
\[
C^{\mu}_{d,n}(z) \leq \int_X \frac{K^{\mu}_{d,n}(x,z)^2}{K^{\mu}_{d,n}(z,z)^2}d\mu(x) = \frac{1}{K^{\mu}_{d,n}(z,z)}.
\]
\end{proof}

\begin{lemma}\label{average}
The CD polynomial has average value
\[
\int_X p^{\mu}_{d,n}(x) d\mu(x) = {n+d\choose n}.
\]
\end{lemma}

\begin{proof}
\begin{align*}
\int_X p^{\mu}_{d,n}(x) d\mu(x) & = 
\int_X b_{d,n}(x)^T(M^{\mu}_{d,n})^{-1}b_{d,n}(x)d\mu(x) \\
& = \text{trace}(M^{\mu}_{d,n})^{-1}\int_X b_{d,n}(x) b_{d,n}(x)^T d\mu(x) \\
& = \text{trace}\:I_{n+d\choose n} = {n+d\choose n}.
\end{align*}
\end{proof}

%

\section{Asymptotic properties}

Let $d\wedge n:=\min(d,n)$.

\begin{lemma}\label{point}
For all $z \in X$, it holds $\lim_{d\wedge n\to\infty} C^{\mu}_{d,n}(z) = \mu(\{z\})$.
\end{lemma}

\begin{proof}
Let $z \in X$.
First observe that $C^{\mu}_{d,n}(z)$ is bounded below and non-increasing i.e. $C^{\mu}_{d'\wedge n'}(z) \leq C^{\mu}_{d \wedge n}(z)$ whenever $d\wedge n \leq d'\wedge n'$, so $\lim_{d\wedge n\to\infty} C^{\mu}_{d,n}(z)$ exists.
If $p$ is admissible for problem \eqref{cf}, it holds
\[
\int_X p^2(x)d\mu(x) \leq p^2(z)\mu(\{z\}) = \mu(\{z\})
\] 
so $\lim_{d\wedge n\to\infty} C^{\mu}_{d,n}(z) \leq \mu(\{z\})$.
Conversely, for given $d,n \in \N$, let
\[
p(x):=(1-|\pi_n(x-z)|^2)^d \in P_{2d,n}
\]
and observe that $p(z)=1$ so that $p(.)$ is admissible for problem \eqref{cf} and
\begin{align*}
C^{\mu}_{2d+1,n}(z) \leq C^{\mu}_{2d,n}(z) & \leq \int_X p(x)^2 d\mu(x)  \\
& \leq \int_{B(z,d^{-\frac14})}d\mu(x) + \int_{X\setminus B(z,d^{-\frac14})} (1-|\pi_n(x-z)|^2)^{2d} d\mu(x)
\end{align*}
where $B(z,r):=\{x \in X : |x-z| \leq r\}$.
For all $x \in X\setminus B(z,d^{-\frac14})$, using \eqref{np} it holds
\[
|x-z|^2=\sum_{k=0}^n \langle x-z, e_k \rangle^2 + \sum_{k=n+1}^{\infty} \langle x-z, e_k \rangle^2  \geq d^{-\frac12}
\]
and hence
\[
(1-|\pi_n(x-z)|^2)^{2d} = (1-\sum_{k=0}^n \langle x-z, e_k \rangle^2)^{2d} \leq 
(1-d^{-\frac12}+\sum_{k=n+1}^{\infty} \langle x-z, e_k \rangle^2)^{2d}
\]
from which it follows that
\[
\lim_{n\to\infty} \int_{X\setminus B(z,d^{-\frac14})} (1-|\pi_n(x-z)|^2)^{2d} d\mu(x) \leq \lim_{n\to\infty} (1-d^{-\frac12}+\sum_{k=n+1}^{\infty} \langle x-z, e_k \rangle^2)^{2d} = (1-d^{-\frac12})^{2d}.
\]
Combining these asymptotic expressions we get
\[
\lim_{d\wedge n\to\infty} C^{\mu}_{d,n}(z) \leq \lim_{d\to\infty} \int_{B(z,d^{-\frac14´})}d\mu(x) + (1-d^{-\frac12})^{2d} = \mu(\{z\}).
\]
\end{proof}

If $\mu$ is absolutely continuous with respect to e.g. the Gaussian measure restricted to $X$ \cite{b98,d06}, then 
it follows from Lemmas \ref{average} and \ref{point} that the Christoffel function on $X$ decreases to zero linearly with respect to the dimension of the vector space $P_{d,n}$. Equivalently, the CD polynomial on $X$ increases linearly with respect to the dimension. This is in sharp contrast with its exponential growth outside of $X$, captured by the following result.

\begin{lemma}\label{outside}
Let $d,n \in \N$.
For all $z \in H$ such that $\min_{x \in X}|\pi_n(x-z)| \geq \delta > 0$ it holds
\[
p^{\mu}_{d,n}(z) \geq 2^{\frac{\delta}{\delta+\text{diam}\:X}d-3}
\]
where $\text{diam}\:X:=\max_{x_1,x_2 \in X} |x_1-x_2|$.
\end{lemma}

\begin{proof}
Let $\delta \in (0,1)$ and let
\[
q(x):=\frac{T_d(1+\delta^2-|\pi_n(x)|^2)}{T_d(1+\delta^2)}
\]
where $T_d$ is the univariate Chebyshev polynomial of the first kind of degree $d \in \N$. This polynomial of $x \in H$ is such that 
\begin{itemize}
\item $q(0)=1$,
\item $|q(x)| \leq 1$ whenever $|\pi_n(x)| \leq 1$,
\item $|q(x)| \leq 2^{1-\delta d}$ whenever $0 < \delta \leq |\pi_n(x)| \leq 1$,
\end{itemize}
see \cite[Lemma 6.3]{lp19}. 
Now let
\[p(x):=q\left(\frac{x-z}{\delta+\text{diam}\:X}\right),
\quad \bar{\delta} := \frac{\delta}{\delta+\text{diam}\:X}
\]
so that if $z \in H$ is such that $\min_{x \in X}|\pi_n(x-z)| \geq \delta > 0$ then \[0 < \bar{\delta} \leq \left|\pi_n\left(\frac{x-z}{\delta+\text{diam}\:X}\right)\right| \leq 1\]
for all $x \in X$. Note that $p(.) \in P_{2d,n}$ and $p(z)=1$ so that $p(.)$ is admissible in problem \eqref{cf} and hence
\[
C^{\mu}_{2d,n}(z) \leq \int_X p(x)^2 d\mu(x)
\leq \int_X (2^{1-\bar{\delta}d})^2 d\mu(x) = 2^{2-2\bar{\delta}d} \leq 2^{3-2\bar{\delta}d}.
\]
Also $C^{\mu}_{2d+1,n}(z) \leq C^{\mu}_{2d,n}(z) \leq  2^{3-2\bar{\delta}d}$ and since $\bar{\delta} < 1$, it holds $C^{\mu}_{2d+1,n}(z) \leq 2^{3-\bar{\delta}(2d+1)}$, from which we conclude that $C^{\mu}_{d,n}(z) \leq 2^{3-\bar{\delta}d}$.
\end{proof}

\section*{Acknowledgements}

These notes benefited from feedback by Nicolas Augier, Francis Bach, Jean Bernard Lasserre, Edouard Pauwels, Alessandro Rudi.

\end{document}